\documentclass[a4paper, 12pt, tikz]{emilsart}

\DeclareMathOperator{\ehr}{ehr}

\newcommand*\circled[2][0.8pt]{%
  \tikz[baseline=(t.base)]{%
    \node[inner sep=0pt, anchor=base] (t) {\(\mathstrut #2\)};%
    \node[draw,circle,inner sep=0pt,minimum size=3ex, yshift=#1] at (t.center) {};%
  }%
}
\newcommand*\emptycircle[1][0.8pt]{%
  \tikz[baseline=(t.base)]{%
    \node[inner sep=0pt, anchor=base] (t) {\(\mathstrut \vphantom{0}\)};%
    \node[draw,circle,inner sep=0pt,minimum size=3ex, yshift=#1] at (t.center) {};%
  }%
}

\usepackage[backend=biber, isbn=false, doi=false, giveninits=true, style=alphabetic, maxbibnames=10]{biblatex}
\title{Interpreting the Ehrhart coefficients of cross-polytopes}
\author{Krishna Menon and Emil Verkama}

\affiliationblock{Krishna Menon, Emil Verkama}{Department of Mathematics, KTH Royal Institute of Technology, Stockholm, Sweden}{puzhan@kth.se, verkama@kth.se}

\begin{document}
 
\maketitle

\begin{abstract}
    \noindent It is known that the Ehrhart polynomials of cross-polytopes, as well as of pyramids over them, have positive coefficients. 
    We give a combinatorial proof of this fact by showing that a scaled version of the Ehrhart polynomials are generating functions for certain colored permutations. 
    This answers a question posed by Stanley.
\end{abstract}

\section{Introduction}\label{sec:intro}

The Ehrhart polynomial \(\ehr_{\mathcal P}(n)\) of a lattice polytope \(\mathcal P\) counts the number of integer lattice points in the dilate \(n \mathcal P\). It is known that the Ehrhart polynomial of the cross-polytope
\begin{equation*}
  \mathcal P_d \coloneqq \left\{\bm x \in \RR^d : |x_1| + |x_2| + \ldots + |x_d| \leq 1\right\}
\end{equation*}
is determined by the generating function
\begin{equation*}
  \sum_{n \geq 0} \ehr_{\mathcal P_d}(n)\, x^n = \frac{(1+x)^d}{(1-x)^{d+1}}.
\end{equation*}
More generally, if \(k \in [0,d]\), the polynomial \(P_{d,k}(n)\) defined by
\begin{equation} \label{eq:polynomial def}
  \sum_{n \geq 0} P_{d,k}(n)\, x^n = \frac{(1+x)^k}{(1-x)^{d+1}}
\end{equation}
is the Ehrhart polynomial of the \((d-k)\)-fold pyramid over \(\mathcal P_k\). The relevant background can be found for example in \cite{ccd}. The figure below illustrates \(\mathcal P_3\) (left), the pyramid over \(\mathcal P_2\) (middle) and \(\mathcal P_2\) itself (right).

\bigskip
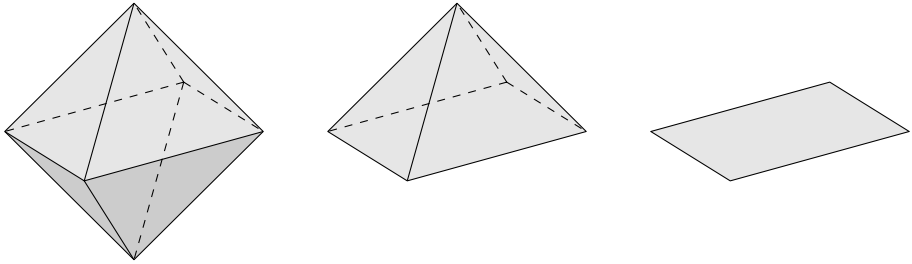
\begin{center}
  \begin{tikzpicture}[scale=1.7]
    \coordinate (1) at (1,0,0);
    \coordinate (-1) at (-1,0,0);
    \coordinate (2) at (0,1,0);
    \coordinate (-2) at (0,-1,0);
    \coordinate (3) at (0,0,1);
    \coordinate (-3) at (0,0,-1);

    \fill[gray!20] (1) -- (3) -- (-1) -- (2) -- cycle;
    \fill[gray!40] (1) -- (3) -- (-1) -- (-2) -- cycle;
    
    \draw[rpath, dashed]
      (-3) -- (1)
      (-3) -- (2)
      (-3) -- (-1)
      (-3) -- (-2);
    \draw[rpath]
      (1) -- (3) -- (-1)
      (2) -- (3) -- (-2);
    \draw[rpath] (1) -- (2) -- (-1) -- (-2) -- cycle;

    \begin{scope}[shift={(2.5,0,0)}]
      \coordinate (1) at (1,0,0);
      \coordinate (-1) at (-1,0,0);
      \coordinate (2) at (0,1,0);
      \coordinate (3) at (0,0,1);
      \coordinate (-3) at (0,0,-1);

      \fill[gray!20] (1) -- (3) -- (-1) -- (2) -- cycle;

      \draw[rpath, dashed]
        (-3) -- (1)
        (-3) -- (-1)
        (-3) -- (2);
      \draw[rpath] (3) -- (2);
      \draw[rpath] (1) -- (3) -- (-1) -- (2) -- cycle;
    \end{scope}

    \begin{scope}[shift={(5,0,0)}]
      \coordinate (1) at (1,0,0);
      \coordinate (-1) at (-1,0,0);
      \coordinate (3) at (0,0,1);
      \coordinate (-3) at (0,0,-1);

      \fill[gray!20] (1) -- (3) -- (-1) -- (-3) -- cycle;
      \draw[rpath] (1) -- (3) -- (-1) -- (-3) -- cycle;
    \end{scope}
  \end{tikzpicture}
\end{center}

\noindent
The fact that \(P_{d,k}(n)\) is a polynomial is not immediately clear from \eqref{eq:polynomial def}, but one can check that
\begin{equation*}
  P_{d,k}(n) = \sum_{i=0}^k \binom{k}{i} \binom{n + d - i}{d}.
\end{equation*}

The coefficients of Ehrhart polynomials have been studied extensively, and they are usually hard to understand -- it is uncommon to be able to fully describe the coefficients for a given polytope. It is difficult even to determine if the coefficients are nonnegative, see \cite{ehrpos}. In this paper, we will examine these questions for (pyramids over) cross-polytopes. We define \(c_{d,k}(j)\) by
\begin{equation} \label{eqn}
  d! P_{d,k}(n) = \sum_{j=0}^d c_{d,k}(j) n^j,
\end{equation}
where the factor \(d!\) ensures that every coefficient is integral. Our main results are a combinatorial interpretation of the coefficients \(c_{d,k}(j)\), and a bijective proof of equation \eqref{eqn}. In particular, we prove combinatorially that the Ehrhart coefficients for (pyramids over) cross-polytopes are positive, which answers a question posed by Stanley \cite[Chapter 4, Exercise 1(b)]{ecnewex}.

The outline of the paper is as follows. 
In Section \ref{sec:sri}, we construct a class of objects that are counted by the coefficients \(c_{d, k}(j)\), in particular showing that these coefficients are positive. 
In Section \ref{sec:bij}, we give a combinatorial argument that proves \eqref{eqn} by constructing objects that are counted by both sides of the equation and exhibiting a bijection. 
In Section \ref{sec:prop}, we use the combinatorial meaning we have for \(c_{d, k}(j)\) to obtain some properties of these coefficients. 
Finally, in Section \ref{sec:poly}, we illustrate how the method we have used might be helpful for studying the Ehrhart coefficients of other polytopes.

\bigskip\noindent
Before proceeding with the general results, we will briefly describe a proof strategy for the simplest case \(k = 0\). The same ideas will serve as our guiding light for the rest of the paper.

The numbers \(c_{d, 0}(j)\) are just unsigned Stirling numbers of the first kind. 
Recall that the (signed) Stirling number of the first kind \(s(m, i)\) is defined by
\begin{equation*}
    s(m, i) \coloneqq (-1)^{m - i} \cdot |\{\sigma \in S_m : \sigma \text{ has } i \text{ cycles}\}|.
\end{equation*}

\begin{proposition}
    For any $d \geq 1$ and $j \in [0, d]$, we have $c_{d, 0}(j) = |s(d + 1, j + 1)|$.
\end{proposition}

\begin{proof}
    From the definition of $P_{d, k}$, we see that $P_{d, 0}(n) = \binom{n + d}{d}$. 
    The result now follows from the fact that $s(m, i)$ is the coefficient of $x^i$ in $x(x - 1) \cdots (x - m + 1)$ (see, for example, \cite[Proposition 1.3.7]{ec1}).
\end{proof}

Although we could stop here and say we have given combinatorial meaning to the coefficients $c_{d, 0}(j)$, we would ideally like to give combinatorial meaning to the fact that for any $n \geq 0$,
\begin{equation*}
    d!P_{d, 0}(n) = \sum_{j = 0}^d c_{d, 0}(j)n^j.
\end{equation*}
Especially since the polynomials $P_{d, k}$ are Ehrhart polynomials, it makes more sense to interpret $c_{d, k}(j)$ in this context rather than doing so independently. 
Also note that if we show that $|s(d + 1, j + 1)|$ satisfies the above equation for all $n \geq 0$, then this implies that $c_{d, 0}(j) = |s(d + 1, j + 1)|$.

To do this, we construct objects counted by $d!P_{d, 0}(n) = d!\binom{n + d}{d}$. 
One choice is objects that consist of a permutation of size $d$ with $n$ balls distributed among the $d + 1$ slots created by the terms of the permutation. 
For example, for $d = 9$ and $n = 6$, such an object is
\begin{equation*}
    \emptycircle\, 281\, \emptycircle\, 9\, \emptycircle\, \emptycircle\, 354\, \emptycircle\, 67\, \emptycircle\,.
\end{equation*}
It can also be shown that these objects are counted by $d!P_{d, 0}(n)$ using the interpretation of $P_{d, 0}(n)$ as an Ehrhart polynomial (see Remark \ref{ehrbij}).

We interpret $|s(d + 1, j + 1)|n^j$ as choosing a permutation of size $d + 1$ with $j + 1$ cycles as well as assigning the labels $1, 2, \ldots, j$ to $n$ balls. 
Such a pair for $d = 8$, $j = 3$, and $n = 4$ is
\begin{equation*}
    (2) (7\ 3\ 5\ 4) (8\ 1) (9\ 6)\quad \text{and}\quad \emptycircle\, \overset{1,3}{\emptycircle}\, \emptycircle\, \overset{2}{\emptycircle}\,.
\end{equation*}

We write cycles in a permutation in \emph{canonical cycle form}: each cycle has its largest element first, and the cycles are ordered in increasing order of their largest elements. 
The labels on balls indicate which cycles must be placed before it. 
The label $1, 3$ on the second ball in the example above indicates that the first and third cycle should be placed before the second ball. 
We do this by placing the cycles as they are into that slot and then removing the brackets. 
This is just an instance of the fundamental bijection on permutations (see \cite[Proposition 1.3.1]{ec1}). 
The cycle containing $d + 1$ is placed after the last ball and then the brackets and $d + 1$ are deleted. 
Hence, the object corresponding to the example above is
\begin{equation*}
    \emptycircle\, 281\, \emptycircle\, \emptycircle\, 7354\, \emptycircle\, 6.
\end{equation*}

It can be checked that this is a bijection and gives us a proof that for all $n \geq 0$,
\begin{equation*}
    d!P_{d, 0}(n) = \sum_{j = 0}^d |s(d + 1, j + 1)| n^j
\end{equation*}
and in particular $c_{d, 0}(j) = |s(d + 1, j + 1)|$. 
We will generalize the interpretation we have for $c_{d, 0}(j)$ in Section \ref{sec:sri} and then in Section \ref{sec:bij}, we will present a bijection generalizing the one we have just seen.

\section{A sign-reversing involution}\label{sec:sri}

In this section, we give a combinatorial meaning to the numbers $c_{d, k}(j)$. 
To do this, we first obtain an expression for $c_{d, k}(j)$ in terms of Stirling numbers of the first kind.

\begin{lemma}\label{lem:str}
    For any $d \geq 1$ and $k, j \in [0, d]$, we have
    \begin{equation*}
        c_{d, k}(j) = \sum_{i = 0}^k \binom{k}{i} \sum_{\ell = 0}^{d - i} |s(d - i + 1, \ell + 1)|s(i, j - \ell).
    \end{equation*}
\end{lemma}

\begin{proof}
    From the expression we have for $P_{d, k}(n)$, we get
    \begin{align*}
        d!P_{d, k}(n) &= \sum_{i = 0}^k \binom{k}{i} (n - i + 1) (n - i + 2) \cdots (n - i + d)\\
        &= \sum_{i = 0}^k \binom{k}{i} (n + 1)(n + 2) \cdots (n + d - i) \,\cdot\, n (n -1) \cdots (n - (i - 1)).
    \end{align*}
    As before, the result now follows from the generating function for the Stirling numbers of the first kind (see \cite[Proposition 1.3.7]{ec1}).
\end{proof}

Using the expression in the lemma above and a sign reversing involution, we will show that the objects defined below are counted by the numbers $c_{d, k}(j)$.

\begin{definition}\label{def:Cdk}
  Let \(C_{d,k}(j)\) denote the set of permutations of size \(d+1\) with \(j+1\) cycles, such that every cycle with all elements less than or equal to \(k\) is odd, and colored red or blue. The other cycles can be even or odd and are left uncolored. We also write
  \begin{equation*}
    C_{d,k} = \bigcup_j C_{d,k}(j).
  \end{equation*}
\end{definition}

\begin{example}
    The following is an element of $C_{14, 10}(4)$:
    \begin{equation*}
        \textcolor{red}{(4)} \textcolor{RoyalBlue}{(7\ 2\ 3)} \textcolor{red}{(10\ 6\ 1)} (14\ 11\ 5\ 12\ 13) (15\ 8).
    \end{equation*}
\end{example}

Just as before, we write permutations in their canonical cycle form: each cycle has its largest element first, and the cycles are ordered in increasing order by their largest elements. 
Hence, when we write an element of $C_{d, k}$, all cycles starting with a number in $[k]$ must be odd and are colored red or blue.

The following theorem answers \cite[Chapter 4, Exercise 1(b)]{ecnewex}.

\begin{theorem} \label{thm:main}
  For every \(d \geq 1\) and \(k, j \in [0, d]\), we have
  \begin{equation*}
    c_{d,k}(j) = |C_{d,k}(j)|
  \end{equation*}
  and hence $c_{d, k}(j)$ is a positive integer.
\end{theorem}

\begin{proof}
    Note that for any permutation $\sigma \in S_m$ with $i$ cycles, the parity of $m - i$ is the same as the number of even cycles in $\sigma$. 
    We now interpret the expression from Lemma \ref{lem:str} for $c_{d, k}(j)$ as the signed count of the following objects: 
    Permutations of size $d + 1$ with $j + 1$ cycles where
    \begin{itemize}
        \item any cycle with all elements in $[k]$ is colored either red or blue,
        \item any cycle containing an element not in $[k]$ is uncolored, and
        \item the sign associated to the permutation is the number of even cycles that are colored red.
    \end{itemize}

    We can view
    \begin{equation*}
        c_{d, k}(j) = \sum_{i = 0}^k \binom{k}{i} \sum_{\ell = 0}^{d - i} |s(d - i + 1, \ell + 1)|s(i, j - \ell)
    \end{equation*}
    as counting these objects by considering that $\binom{k}{i}$ and $s(i, j - \ell)$ in the expression choose the red cycles. 
    The remaining cycles are chosen by $|s(d - i + 1, \ell + 1)|$ and we color blue such cycles that only contain elements in $[k]$. 
    The reason that we color these cycles blue is just to make defining the sign-reversing involution more convenient.

    We now pair each such object that has a negative sign with an object that has a positive sign. 
    As indicated by the definition of $C_{d, k}$, we will define the sign-reversing involution on those objects that have at least one even cycle that is colored. 
    The involution is defined as follows: 
    When the permutation is written in canonical cycle form, find the first even cycle that is colored, and change its color.

    For example, an instance of this involution is
    \begin{gather*}
        \textcolor{RoyalBlue}{(6\ 4\ 3)} \textcolor{red}{(9\ 2)} \textcolor{red}{(11\ 7)} \textcolor{RoyalBlue}{(10\ 1)} (13\ 5\ 8\ 12)\\
        \updownarrow\\
        \textcolor{RoyalBlue}{(6\ 4\ 3)} \textcolor{RoyalBlue}{(9\ 2)} \textcolor{red}{(11\ 7)} \textcolor{RoyalBlue}{(10\ 1)} (13\ 5\ 8\ 12).
    \end{gather*}
    
    It is clear that this is an involution and since the number of red even cycles changes exactly by $1$, it is a sign-reversing involution. 
    Any object with a negative sign must have at least one red cycle that is even and is hence taken care of by this involution. 
    The only objects that remain are those in $C_{d, k}(j)$, which proves the result.
\end{proof}

For the case \(k = d\), using the definition of \(C_{d, d}(j)\) and properties of exponential generating functions (for example, see \cite[Chapter 5]{ec2}), it can be shown that
\begin{equation*}
    \sum_{d \geq 0} c_{d, d}(j) t^j \frac{x^d}{d!} = \left(\sum_{i \geq 0}x^i\right) \left(\sum_{i \geq 0}\frac{2x^{2i + 1}}{2i + 1}\right)^t.
\end{equation*}
The first factor on the right hand side chooses the cycle containing $d + 1$. 
The second factor partitions the remaining numbers into blocks of odd sizes and gives each block the structure of a red or blue cycle. 
This expression matches with the one computed by McKay on MathOverflow \cite{overflow}.

\section{The bijection}\label{sec:bij}

In this section, we prove \eqref{eqn} by constructing objects that are counted by both sides of the equation and exhibiting a bijection. 
In particular, this will give a bijective proof of Theorem \ref{thm:main}.

We start with a combinatorial interpretation of the the values \(d! P_{d,k}(n)\). Let \(W_{d,k}(n)\) be the collection of words in the alphabet
\begin{equation*}
  \emptycircle\,,\, \circled{1}\,,\, \circled{2}\,,\, \ldots,\, \circled{k}\,,\, 1,\, 2,\, \ldots,\,  d
\end{equation*}
subject to the following conditions:
\begin{itemize}
  \item Each number \(1, 2, \ldots, d\) appears exactly once (either circled or uncircled).
  \item There are exactly \(n\) circles (empty or with a number inside).
\end{itemize}
In other words, the elements of \(W_{d,k}(n)\) are words consisting of permutations of size \(d\) together with \(n\) circles, so that a circle can be written between two other entries or around an entry \(1, \ldots, k\) of the permutation. It is clear from the generating function of $P_{d, k}(n)$ that \(|W_{d,k}(n)| = d! P_{d,k}(n)\).

\begin{example}
  \(W_{3,1}(1)\) consists of the words
  \begin{equation*}
    123\, \emptycircle\,, \enskip 12\, \emptycircle\, 3, \enskip 1\, \emptycircle\, 23, \enskip \emptycircle\, 123, \enskip \circled{1}\, 23, \enskip 213\, \emptycircle\,, \enskip 2\, \circled{1}\, 3, \ldots.
  \end{equation*}
  It does not contain words such as \(1\, \circled{2}\, 3\), because \(2 > k\) and cannot be circled.
\end{example}

\begin{remark}\label{ehrbij}
    The fact that the objects described above are counted by $d!P_{d, k}(n)$ also follows by the interpretation of $P_{d, k}(n)$ as an Ehrhart polynomial. 
    More specifically, $P_{d, k}(n)$ counts points $\bm x \in \ZZ^d$, such that
    \begin{itemize}
        \item $|x_1| + \ldots +|x_k| + x_{k + 1} + \ldots + x_d \leq n$, and
        \item $x_{k + 1}, \ldots, x_d \geq 0$.
    \end{itemize}
    One way to represent a sequence of length $d$ of integers with sum of absolute values at most $n$ is as a sequence consisting of $n$ balls and $d$ dots. 
    For example, with $n = 6$ and $d = 3$, the sequence $(2, 0, 3)$ (with only nonnegative terms) is represented as
    \begin{equation*}
        \emptycircle[0pt]\, \emptycircle[0pt]\, {\bullet}\, {\bullet}\, \emptycircle[0pt]\, \emptycircle[0pt]\, \emptycircle[0pt]\, {\bullet}\, \emptycircle[0pt].
    \end{equation*}
    The total number of balls is $n$, the first entry is the number of balls before the first dot, and for any $i \geq 2$, the $i$-th entry is the number of balls between the $(i - 1)$-th and $i$-th dots. 
    If we want the $i$-th entry to be negative, we move the $i$-th dot into the ball immediately before it. 
    For example, the sequence $(-1, 3, -2)$ is represented as
    \begin{equation*}
        \circled[0pt]{\bullet}\, \emptycircle[0pt]\, \emptycircle[0pt]\, \emptycircle[0pt]\, {\bullet}\, \emptycircle[0pt]\, \circled[0pt]{\bullet}.
    \end{equation*}
    Hence, $P_{d, k}(n)$ counts such configurations with $n$ balls and $d$ dots where only the first $k$ dots can be in balls. 
    This interpretation can be used to show $d!P_{d, k}(n) = |W_{d, k}(n)|$.
\end{remark}

Our goal is to provide a combinatorial proof of Theorem \ref{thm:main}. The building blocks of our bijection are restricted versions of the \(n = 1\) case. The idea is that if we have a word such as
\begin{equation*}
  \circled{8}\, \emptycircle\, 517\, \circled{3}\, 4 \, \emptycircle\, 26
\end{equation*}
in \(W_{8,8}(4)\), we can treat the blocks
\begin{equation*}
  \circled{8}\,, \enskip \emptycircle\,,\enskip 517\,\circled{3}\,,\enskip 4\,\emptycircle\, \enskip \text{and} \enskip 26
\end{equation*}
separately and independently to obtain partial permutations of \(\{8\}\), \(\{5,1,7,3\}\), \(\{4\}\) and \(\{2,6,9\}\) (where \(9 = d+1\)) satisfying the odd, colored cycle condition of Definition \ref{def:Cdk}. The composition of these partial permutations will then give us an element of \(C_{8,8}\). Just as in the $k = 0$ case presented in Section \ref{sec:intro}, an element of \([n]^j\) will tell us which circle of the word is associated to each cycle in the final permutation (except the one containing \(d+1\), which will be placed last), making the correspondence bijective. We will explain this in more detail later.

Recall the \emph{fundamental bijection} \(S_d \to S_d\) defined by taking a permutation \(\pi\) in canonical cycle form and then removing the parentheses to obtain a permutation \(\widehat \pi\) written in one line notation. We will use this bijection extensively.

\begin{lemma} \label{lem:colored cycles to marked singleton}
  The collection \(C_{d,k}\) is in bijection with the set of all permutations of size \(d+1\) such that one singleton cycle \((i)\) with \(i \leq k\) may be colored red.
\end{lemma}

\begin{remark}
  This set is enumerated by \((d+1)! + k d! = (d+k+1) d!\). The set is, furthermore, in bijection with the words in \(W_{d,k}(1)\) such that there are no numbers after the circle: if there is a red singleton cycle \((i)\), write \(i\) inside the circle and apply the fundamental bijection to the rest of the permutation. Then write the resulting permutation in one line notation before the circle. Otherwise, apply the fundamental bijection to the entire permutation and write it in one line notation before the circle.
\end{remark}

Our proof of Lemma \ref{lem:colored cycles to marked singleton} is a simple modification of the argument of Bóna, McLennan and White proving that the number of permutations of size \(2n\) with all cycles odd equals the number of permutations of size \(2n\) with all cycles even \cite{bona_permutations_2000}.

\begin{proof}[Proof of Lemma \ref{lem:colored cycles to marked singleton}]
  Let \(\pi \in C_{d,k}\). We leave the uncolored cycles of \(\pi\) unchanged. Then uncolor the blue cycles, otherwise leaving them unchanged. The red cycles \(\sigma_1, \ldots, \sigma_{\ell}\) are ordered in canonical cycle form: each cycle has its largest element first, and the cycles are ordered in increasing order by their largest elements. Suppose first that there is an even number \(\ell = 2m\) of red cycles. For each \(1 \leq i \leq m\), remove the last element of \(\sigma_{2i-1}\) and place it in the last position of \(\sigma_{2i}\). Uncolor the red cycles. If there is an odd number \(\ell = 2m+1\) of red cycles, apply the same procedure to \(\sigma_1, \ldots, \sigma_{2m}\), remove the last element of \(\sigma_{2m+1}\) and place it into its own singleton cycle. We uncolor all cycles, except for the new singleton cycle, which stays red. See Example \ref{ex:colored cycles to marked singleton}.

  To see that this is a bijection, we describe the inverse map. Let \(\sigma\) be a permutation of size \(d+1\), possibly with one red cycle \((i)\). We leave the cycles of \(\sigma\) containing entries greater than \(k\) unchanged. Uncolored odd cycles with all elements less than or equal to \(k\) are colored blue. What remains is a collection of even cycles with all elements less than or equal to \(k\), and possibly one red singleton cycle. We organize the cycles in canonical cycle form, except for the singleton which is placed last. If there is such a singleton, we remove it and place its element in the last position of the preceding cycle. This cycle is colored red. For the remaining cycles \(\tau_1, \ldots, \tau_h\) (still in canonical cycle form) we apply the following procedure:
  \begin{itemize}
      \item If the last entry of \(\tau_h\) is larger than the first entry of \(\tau_{h-1}\), remove the last entry of \(\tau_h\), place it into its own singleton cycle, and repeat the whole procedure for the remaining cycles \(\tau_1, \ldots, \tau_{h-1}\).
      \item If the last entry of \(\tau_h\) is smaller than the first entry of \(\tau_{h-1}\), remove it and place it in the last position of \(\tau_{h-1}\), and repeat the procedure for the remaining cycles \(\tau_1, \ldots, \tau_{h-2}\).
  \end{itemize} 
  We color all the resulting cycles red.
\end{proof}

\begin{example} \label{ex:colored cycles to marked singleton}
  Consider the permutation
  \begin{equation*}
    \textcolor{red}{(5\ 1\ 3)(6)} \textcolor{RoyalBlue}{(9\ 2\ 7)} \textcolor{red}{(13\ 8\ 12\ 10\ 11)} (16\ 14\ 15\ 4) \in C_{16,13}.
  \end{equation*}
  The blue and uncolored cycles are left unchanged and we map
  \begin{equation*}
    \textcolor{red}{(5\ 1\ 3)(6)(13\ 8\ 12\ 10\ 11)} \longmapsto (5\ 1)(6\ 3)(13\ 8\ 12\ 10) \textcolor{red}{(11)},
  \end{equation*}
  so the resulting permutation is
  \begin{equation*}
    (5\ 1)(6\ 3)(9\ 2\ 7)\textcolor{red}{(11)}(13\ 8\ 12\ 10)(16\ 14\ 15\ 4).
  \end{equation*}
\end{example}

We are ready to give a combinatorial proof of Theorem \ref{thm:main}.

\begin{proof}[Proof of Theorem \ref{thm:main}]
  We want to construct a bijection
  \begin{equation*}
    \bigcup_{j = 0}^d C_{d,k}(j) \times [n]^j \longrightarrow W_{d,k}(n).
  \end{equation*}
  We interpret the left-hand side as follows: each element of \(\pi \in C_{d,k}(j)\) is written in canonical cycle form. An element of \([n]^j\) assigns labels \(1, \ldots, j\) arbitrarily to a sequence of \(n\) circles, so that a circle can receive zero, one or several labels. The first \(j\) cycles of our permutation \(\pi\) are assigned to the circles according to these labels: the \(i\)-th cycle is assigned to the circle labeled \(i\). 
  
  We now have a sequence of pairs, each consisting of a circle and a (possibly empty) partial permutation of \(\pi\). Within each pair, we apply the bijection of Lemma~\ref{lem:colored cycles to marked singleton} to the partial permutation to obtain an uncolored permutation \(\sigma\) of the same entries, possibly with one singleton cycle colored red. If there is no red singleton, the pair is mapped to the word consisting of \(\widehat \sigma\) written in one line notation followed by an empty circle. If there is a red singleton \(i\), then the pair is mapped to the word consisting of \(\widehat \sigma\) written in one line notation with entry \(i\) removed, followed by a circle containing \(i\).

  To proceed, we concatenate the \(n\) words obtained in this way. To finish the construction, we apply the fundamental bijection to the last cycle of \(\pi\), write it in one line notation, remove the entry \(d+1\), and append it to our word. The result is an element of \(W_{d,k}(n)\).

  This construction is bijective. To prove this, we will now describe its inverse (see also Example \ref{ex:bijection}). Starting from an element of \(w \in W_{d,k}(n)\), take the (possibly empty) suffix of numbers after the last circle, prepend \(d+1\), and apply the inverse of the fundamental bijection to obtain a single cycle. The remaining part of \(w\) can be split into \(n\) blocks, each ending with a circle. Apply the inverse of the fundamental bijection to the numbers before the circle in each block. If the circle contains a number \(i\), place \(i\) into its own red singleton cycle in the resulting permutation. Now, apply the inverse of the bijection of Lemma~\ref{lem:colored cycles to marked singleton} to the partial permutations in each block. The composition of all resulting cycles gives us an element \(\pi \in C_{d,k}\). If the \(i\)-th cycle in the canonical cycle form of \(\pi\) came from the \(\ell\)-th block of \(w\), assign label \(i\) to the \(\ell\)-th circle. The pair consisting of \(\pi\) and the labeling of circles is the preimage of \(w\).
\end{proof}

\begin{example} \label{ex:bijection}
  Let 
  \begin{equation*}
    w \ = \ \circled{8}\, \emptycircle\, 517\, \circled{3}\, 4 \, \emptycircle\, 26 \ \in \ W_{8,8}(4)
  \end{equation*}
  The suffix of numbers gives us the cycle \((9\ 2\ 6)\). The remaining blocks are
  \begin{equation*}
    \left(\,\circled{8}\,,\ \emptycircle\,,\ 517\,\circled{3}\,,\ 4\,\emptycircle\,\right).
  \end{equation*}
  Applying the inverse of the fundamental bijection to each block, and adding a red singleton cycle consisting of the entry contained in the circle (if there is one), gives
  \begin{equation*}
    \left(\textcolor{red}{(8)},\ \epsilon,\ (5\ 1)(7)\textcolor{red}{(3)},\ (4)\right),
  \end{equation*}
  where \(\epsilon\) is the empty permutation. Applying the inverse of the bijection of Lemma~\ref{lem:colored cycles to marked singleton} to each partial permutation yields
  \begin{equation*}
    \left(\textcolor{red}{(8)},\ \epsilon,\ \textcolor{red}{(5\ 1\ 3)} \textcolor{RoyalBlue}{(7)},\ \textcolor{RoyalBlue}{(4)}\right).
  \end{equation*}
  After composing these cycles we get
  \begin{equation*}
    \pi \ = \ \textcolor{RoyalBlue}{(4)} \textcolor{red}{(5\ 1\ 3)} \textcolor{RoyalBlue}{(7)} \textcolor{red}{(8)} (9\ 2\ 6) \ \in \ C_{8,8}(4),
  \end{equation*}
  and the element of \([n]^j\) keeping track of which cycle of \(\pi\) is assigned to which circle is \((4, 3, 3, 1)\). In conclusion, \(w\) is mapped to the pair
  \begin{equation*}
    \textcolor{RoyalBlue}{(4)} \textcolor{red}{(5\ 1\ 3)} \textcolor{RoyalBlue}{(7)} \textcolor{red}{(8)} (9\ 2\ 6) \quad \text{and} \quad
    \overset{4}{\emptycircle}\, \emptycircle\, \overset{2,3}{\emptycircle}\, \overset{1}{\emptycircle}.     
  \end{equation*}
\end{example}

\section{Properties of the coefficients}\label{sec:prop}

We now obtain some results about the numbers \(c_{d, k}(j)\) using the combinatorial interpretations and bijections from previous sections. Some of these results also follow from known properties of the Ehrhart polynomials, but the combinatorial perspective might be interesting.

\begin{lemma}
    For any $d \geq 1$ and $k \in [0, d]$, we have $c_{d, k}(0) = d!$ and $c_{d, k}(d) = 2^k$.
\end{lemma}

\begin{proof}
    The elements in $C_{d, k}(0)$ are permutations of size $d + 1$ with only one cycle, and the elements of $C_{d, k}(d)$ are permutations of size $d + 1$ with $d + 1$ cycles, where the \(k\) cycles whose elements are in $[k]$ are colored red or blue.
\end{proof}

\begin{remark}
    The result above also follows from the fact that the constant term of the Ehrhart polynomial of a polytope is $1$ and that the leading term is its volume (see \cite[Corollary 3.20]{ccd}).
\end{remark}

\begin{lemma}
    For any $d \geq 1$ and $k, j \in [0, d]$, we have $c_{d, k}(j) \leq c_{d + 1, k}(j)$.
\end{lemma}

\begin{proof}
    Start with any element $\pi \in C_{d, k}(j)$ and add $d + 2$ to the cycle containing $d + 1$, immediately before $d + 1$. 
    This defines an injection from $C_{d, k}(j)$ into $C_{d + 1, k}(j)$.
\end{proof}

\begin{proposition}
  For every \(d \geq 1\), \(k \leq d-1\) and \(j\) we have that \(c_{d,k}(j) \leq c_{d,k+1}(j)\).
\end{proposition}

\begin{proof}
  We will construct an injection \(C_{d,k}(j) \to C_{d,k+1}(j)\). Let \(\pi \in C_{d,k}(j)\). If \(k+1\) is contained in a cycle with some element greater than \(k+1\), leave \(\pi\) unchanged. If \(k+1\) is contained in a cycle with all elements less than or equal to \(k+1\) and the cycle is odd, then color the cycle blue. If the cycle is even, remove the last element (when written in canonical cycle form), set it as the last element of the cycle containing \(d+1\), and color the cycle containing \(k+1\) red. This is an injection, since the color of the cycle containing \(k+1\) tells us which case we are in.
\end{proof}

\begin{remark}
  The image of this mapping consists of all permutations in \(C_{d,k+1}(j)\) such that \(k+1\) is in a cycle with some element greater than \(k+1\), or \(k+1\) is in a blue cycle, or \(k+1\) is in a red cycle and the element before \(d+1\) in its cycle is less than \(k+1\). This means that the difference \(c_{d,k+1}(j) - c_{d,k}(j)\) counts the permutations in \(C_{d,k+1}(j)\) where \(k+1\) is in a red cycle and the element before \(d+1\) in its cycle is greater than \(k+1\).
\end{remark}

\begin{proposition} \label{prop:alternating sum k<d}
    For any $d \geq 1$ and $k \in [0, d - 1]$, we have
    \begin{equation*}
        \sum_{j = 0}^d (-1)^j c_{d, k}(j) = 0.
    \end{equation*}
\end{proposition}

\begin{proof}
    The proof is very similar to the analogous result for Stirling numbers of the first kind. 
    We define a sign-reversing involution on the objects in $C_{d,k}$ as follows: 
    Since $k < d$, both $d$ and $d + 1$ are in uncolored cycles. 
    If $d$ and $d + 1$ are in the same cycle, we match it with the object obtained by `breaking' the cycle before $d$ (permutation written in canonical cycle form). 
    If $d$ and $d + 1$ are in different cycles, we combine them by moving all the terms in the cycle with $d$ to the end of the cycle with $d + 1$. 
    An instance of this involution with $d = 14$ is
    \begin{gather*}
        \textcolor{red}{(4)} \textcolor{RoyalBlue}{(7\ 2\ 3)} \textcolor{red}{(10\ 6\ 1)} (12\ 5\ 11) (14 \ 13) (15\ 8)\\
        \updownarrow\\
        \textcolor{red}{(4)} \textcolor{RoyalBlue}{(7\ 2\ 3)} \textcolor{red}{(10\ 6\ 1)} (12\ 5\ 11) (15\ 8\ 14 \ 13). \qedhere
    \end{gather*}
\end{proof}

\begin{proposition} \label{prop:alternating sum k=d}
  For every \(d \geq 1\), we have
  \begin{equation*}
    \sum_{j=0}^d (-1)^{d-j} c_{d,d}(j) = d!.
  \end{equation*}
\end{proposition}

\begin{proof}
  Associate the sign \((-1)^{d-j}\) to each object in \(C_{d,d}(j)\). These objects have \(j\) colored cycles, all of odd length, and one uncolored cycle (the one containing \(d+1\)). Hence, the sign of an object is determined by the parity of the length of the cycle containing \(d+1\). Since the objects in \(C_{d,d}(d)\) are positive, this means that an object has positive sign if and only if the cycle containing \(d+1\) is odd.

  Apply the bijection of Lemma \ref{lem:colored cycles to marked singleton} to \(C_{d,d}(j)\) to obtain the set \(A\) of permutations of size \(d+1\) in which one singleton cycle \((i)\) with \(i \leq d\) may be colored red. The bijection does not change the cycle containing \(d+1\), so the sign of each object can still be determined in the same way. We construct a sign-reversing involution defined on the objects \(\pi \in A\) as follows.
  \begin{itemize}
    \item Suppose \(\pi\) does not have a red singleton and \(d+1\) is not a fixed point. Denote the cycle containing \(d+1\) by \((a_1\ \ldots\ a_m)\), where \(a_1 = d+1\). Split the cycle into \((a_1\ \ldots\ a_{m-1})\textcolor{red}{(a_m)}\).
    \item Suppose \(\pi\) has a red singleton \(\textcolor{red}{(i)}\). Denote the cycle containing \(d+1\) by \((a_1\ \ldots\ a_{m-1})\), where \(a_1 = d+1\). Merge the red singleton into the cycle to obtain \((a_1\ \ldots\ a_{m-1}\ i)\).
  \end{itemize}
  This mapping is clearly involutive, and it is sign-reversing since the length of the cycle containing \(d+1\) changes by one. The fixed points of the involution are the permutations in \(A\) such that \(d+1\) is a fixed point and there is no red singleton. There are \(d!\) such permutations and they all have positive sign, so the result follows.
\end{proof}

\begin{example}
  For \(d = 2\), the objects in \(C_{2,2}\) are
  \begin{align*}
    & (3\ 1\ 2) & & \textcolor{red}{(1)}(3\ 2) & & \textcolor{red}{(1)(2)}(3) \\
    & (3\ 2\ 1) & & \textcolor{RoyalBlue}{(1)}(3\ 2) & & \textcolor{red}{(1)}\textcolor{RoyalBlue}{(2)}(3) \\
    & & & \textcolor{red}{(2)}(3\ 1) & & \textcolor{RoyalBlue}{(1)}\textcolor{red}{(2)}(3) \\
    & & & \textcolor{RoyalBlue}{(2)}(3\ 1) & & \textcolor{RoyalBlue}{(1)(2)}(3).
  \end{align*}
  The first and third columns have positive sign, while the second column has negative sign. After applying the bijection of Lemma \ref{lem:colored cycles to marked singleton}, the objects become
  \begin{align*}
    & (3\ 1\ 2) & & \textcolor{red}{(1)}(3\ 2) & & (2\ 1)(3) \\
    & (3\ 2\ 1) & & (1)(3\ 2) & & \textcolor{red}{(1)}(2)(3) \\
    & & & \textcolor{red}{(2)}(3\ 1) & & (1)\textcolor{red}{(2)}(3) \\
    & & & (2)(3\ 1) & & (1)(2)(3).
  \end{align*}
  The sign-reversing involution in the proof of Proposition \ref{prop:alternating sum k=d} matches
  \begin{align*}
    (3\ 1\ 2) &\longleftrightarrow \textcolor{red}{(2)}(3\ 1) & (3\ 2\ 1) &\longleftrightarrow \textcolor{red}{(1)}(3\ 2) \\
    (1)(3\ 2) &\longleftrightarrow (1)\textcolor{red}{(2)}(3) & (2)(3\ 1) &\longleftrightarrow \textcolor{red}{(1)}(2)(3),
  \end{align*}
  so the remaining objects are \((1)(2)(3)\) and \((2\ 1)(3)\).
\end{example}

\begin{remark}
    Propositions \ref{prop:alternating sum k<d} and \ref{prop:alternating sum k=d} are also a consequence of Ehrhart--Macdonald reciprocity (see \cite[Theorem 4.1]{ccd}) applied to (pyramids over) the cross-polytope. 
    This is because the cross-polytope has only one interior lattice point (the origin) and pyramids over them have none.
\end{remark}

\section{Other polytopes?}\label{sec:poly}

In this section, we mention how the methods we have used to study the Ehrhart polynomials of cross-polytopes could be used for other polytopes. We first recall some basic facts from Ehrhart theory; details can be found, for example, in \cite{ccd}. 
A \emph{lattice polytope} in $\RR^d$ is the convex hull of a finite subset of the lattice $\ZZ^d$. 
The dimension of a polytope is its affine dimension. 
The number of lattice points in dilates of a lattice polytope of dimension $d$ is counted by a polynomial of degree $d$. 
This is called its \emph{Ehrhart polynomial}.

For a lattice polytope $\mathcal P$ in $\RR^d$, we use $\ehr_{\mathcal P}$ to denote its Ehrhart polynomial. 
Hence, for any integer $n \geq 0$,
\begin{equation*}
    \ehr_{\mathcal P}(n) = |n\mathcal P \cap \ZZ^d|.
\end{equation*}
If $\mathcal P$ has dimension $d$, its \emph{$h^*$-polynomial} is the unique polynomial $h_{\mathcal P}^*(x)$ of degree at most $d$ such that
\begin{equation*}
    \sum_{n \geq 0}\ehr_{\mathcal P}(n)x^n = \frac{h^*_{\mathcal P}(x)}{(1 - x)^{d + 1}}.
\end{equation*}
For their coefficients, we write
\begin{equation*}
    \ehr_{\mathcal P}(n) = \sum_{j = 0}^d c_{\mathcal P}(j) n^j \quad \text{and} \quad h^*_{\mathcal P}(x) = \sum_{i = 0}^d c^*_{\mathcal P}(i) x^i,
\end{equation*}
respectively. 
For any $0 \leq i, j \leq d$, we set
\begin{equation*}
    e(d, i, j) \coloneqq \sum_{\ell = 0}^{d - i}|s(d - i + 1, \ell + 1)| s(i, j - \ell).
\end{equation*}
Note that these numbers need not be positive. 
Using the same computations as in Lemma \ref{lem:str}, the following observation follows.

\begin{proposition}
    For any \(d\)-dimensional lattice polytope \(\mathcal P\) and any \(j \in [0,d]\), we have that  
    \begin{equation}\label{polyeq}
        d! c_{\mathcal P}(j) = \sum_{i = 0}^d c^*_{\mathcal P}(i) \cdot e(d,i,j).
    \end{equation}
\end{proposition}

For cross-polytopes and pyramids over them, we have used this expression to give combinatorial meaning to the coefficients of their Ehrhart polynomials. 
It might be interesting to see what expressions we obtain for other polytopes.

For example, the unit cube in $\RR^d$ has Ehrhart polynomial $(n + 1)^d$ and $h^*$-polynomial $\sum_{i = 0}^{d - 1} A(d, i) x^i$ where $A(d, i)$ is the Eulerian number counting the permutations in $S_d$ with $i$ descents. 
Plugging this into \eqref{polyeq}, we get that for any $j \in [0, d]$,
\begin{equation*}
    d! \cdot \binom{d}{j} = \sum_{i = 0}^{d - 1} A(d, i) \cdot e(d, i, j)
\end{equation*} 
Although it is not difficult to give the binomial coefficients a combinatorial meaning, it would still be interesting to prove the above equation combinatorially.

A polytope is said to be \emph{Ehrhart positive} if the coefficients of its Ehrhart polynomial are positive. 
For example, cross-polytopes, pyramids over them, and unit cubes are all Ehrhart positive. 
We refer the reader to \cite{ehrpos} for a survey on Ehrhart positivity. 
For polytopes where the \(h^*\)-polynomial has a combinatorial interpretation, the methods used in this paper might help in finding combinatorial proofs for Ehrhart positivity.

\subsection*{Acknowledgements}

KM is supported by the Verg Foundation. EV is supported by the Swedish Research Council grant 2022-03875. The authors are grateful to Lobezno Vecchi for helpful discussions.

\printbibliography

\end{document}